\documentclass{article}
\pdfoutput=1

\usepackage[utf8]{inputenc}
\usepackage{amssymb,amsmath,amsthm,graphicx,amscd,thmtools,tikz}
\usepackage{mathtools,mathscinet}
\usepackage{microtype}
\usepackage{hyperref}
\usepackage[a4paper]{geometry}
\usepackage[nameinlink]{cleveref}

\newtheorem{thm}{Theorem}[section]
\newtheorem{lm}[thm]{Lemma}

\newtheorem{prop}[thm]{Proposition}
\newtheorem*{thm*}{Theorem}
\newtheorem*{lm*}{Lemma}

\theoremstyle{definition}

\DeclareMathOperator{\rank}{rank}
\DeclareMathOperator{\brank}{brank}
\DeclareMathOperator{\spa}{span}

\DeclareMathOperator{\conv}{conv}
\DeclareMathOperator*{\argmax}{arg\,max}

\DeclareMathOperator{\Sym}{Sym}

\newcommand{\abs}[1]{{\left\lvert #1 \right\rvert}}

\DeclareMathOperator{\R}{\mathbb{R}}
\DeclareMathOperator{\C}{\mathbb{C}}

\title{
\bfseries
Maximum relative distance between real rank-two and rank-one tensors}
\author{Henrik Eisenmann\thanks{Max Planck Institute for Mathematics in the Sciences, 04103 Leipzig, Germany} 
\thanks{Corresponding Author. henrik.eisenmann@mis.mpg.de} \and Andr\'e Uschmajew${}^{\ast}$}
\date{}

\begin{document}

\maketitle
\begin{abstract}
 It is shown that the relative distance in Frobenius norm of a real symmetric order-$d$ tensor of rank two to its best rank-one approximation is upper bounded by $\sqrt{1-(1-1/d)^{d-1}}$. This is achieved by determining the minimal possible ratio between spectral and Frobenius norm for symmetric tensors of border rank two, which equals $\left(1-{1}/{d}\right)^{(d-1)/{2}}$. These bounds are also verified for  arbitrary real rank-two tensors by reducing to the symmetric case.
\end{abstract}

\section{Introduction}
It is a well-known fact that the minimal possible ratio between spectral and Frobenius norm of a real $n \times n$ matrix is $1 / \sqrt{n}$, and is achieved for any matrix with identical singular values, that is, for multiples of orthogonal matrices. Since the spectral norm of a matrix measures
the length of its best rank-one approximation, this statement has the geometric meaning that orthogonal matrices achieve the largest possible relative distance to rank-one matrices. More generally, using singular value decomposition, one can show that the minimal ratio between spectral and Frobenius norm of a rank-$k$ matrix is $1/\sqrt{k}$ and is achieved when all  nonzero singular values are equal.

There has been considerable interest in determining the minimal possible ratio between spectral norm $\| A \|_\sigma$  and Frobenius norm $\| A \|_F$ of an $n_1 \times \dots \times n_d$ tensor $A$; see, e.g.,~\cite{Qi_best_11,KM_bounds_15,derksen2017theoretical,LNSU_orthogonal_18,Li2020,AKU_Cheb_20}. As in the matrix case, this ratio measures the distance of 
$A$ to the set of rank-one tensors, and is hence of both theoretical and practical relevance in problems of  low-rank approximation and entanglement. The precise relation between the spectral norm of $A$ and its distance to rank-one tensors is as follows:
\begin{equation}\label{eq:distance}
\min_{\rank B \le 1} \frac{\| A - B \|_F}{\| A \|_F} =\sqrt{ 1 - \frac{\| A \|^2_\sigma}{\| A \|_F^2}}.
\end{equation}
Therefore, the minimal possible ratio $\| A \|_\sigma / \| A \|_F$ that can be achieved is also called the best rank-one approximation ratio of the given tensor space~\cite{Qi_best_11}. By~\eqref{eq:distance}, it expresses the maximum relative distance of a tensor to the set of rank-one tensors.

Despite some recent progress achieved in the aforementioned references and others, determining the best rank-one approximation ratio for tensors remains a difficult problem in general and is largely open. One reason is

the lack of a suitable analog to the singular value decomposition. Moreover, the best rank-one approximation ratio of tensors usually differs over the real and complex field, as well as for nonsymmetric and symmetric tensors of the same size.

The available results in the literature focus on the best rank-one approximation ratio in the full tensor space. As for matrices, it would however also be useful to estimate its value in dependence of the tensor rank. In this work we take a first step in this direction. We determine the minimal ratio between spectral and Frobenius norm of real rank-two tensors, and obtain that it is actually the same for symmetric and general tensors. Recall that for matrices this value equals $1/\sqrt{2}$. 

For tensors one should also take into account that the set of tensors of rank at most two is not closed. Our main result is on symmetric tensors  and reads as follows.

\begin{thm}\label{thm:main}
Let $A$ be a real symmetric tensor of order $d\geq 3$ and rank at most two. Then
\begin{equation}\label{eq: rank two inequality}
 \|A\|_\sigma>\left( 1 - \frac{1}{d} \right)^{\frac{d-1}{2}}\|A\|_F
\end{equation}
and this bound is sharp. In particular,
\[
\min_{\substack{A \neq 0 \\ \brank A  \le 2} } \frac{\| A \|_\sigma}{\|A\|_F} = \left( 1 - \frac{1}{d} \right)^{\frac{d-1}{2}},
\]
where $\brank$ denotes border rank, and the minimum is taken over real symmetric tensors. Up to orthogonal transformation and scaling the minimum is achieved only for the tensor
\[
W_d = \lim_{t \to 0}\frac{1}{t} \left[ (e_1+te_2)^d-e_1^d \right] = d e_1^{d-1} e_2.
\]
\end{thm}
Here $e_1,e_2$ are two orthogonal unit tensors, $u^d$ abbreviates $u \otimes \dots \otimes u$ ($d$ times) and $u^{d-1}v$ denotes the symmetric part of $u^{d-1} \otimes v$ (see below for notation).

The proof of Theorem~\ref{thm:main} constitutes the main part of this work and will be given in section~\ref{sec: proof}. 
The result however raises the question,
whether the same bounds hold for general nonsymmetric tensors of rank two. In section~\ref{sec: nonsymmetric} we show that the answer is affirmative by reducing the question to the symmetric case.

\begin{thm}\label{thm:main2}
Let $A$ be a real $n_1 \times \dots \times n_d$ tensor of rank at most two. Then
\begin{equation}
 \|A\|_\sigma>\left( 1 - \frac{1}{d} \right)^{\frac{d-1}{2}}\|A\|_F
\end{equation}
and this bound is sharp. In particular, assuming $n_i \ge 2$ for $i=1,\dots,d$,
\[
\min_{\substack{A \neq 0 \\ \brank A  \le 2} } \frac{\| A \|_\sigma}{\|A\|_F} = \left( 1 - \frac{1}{d} \right)^{\frac{d-1}{2}},
\]
where $\brank$ denotes border rank, and the minimum is taken over real $n_1 \times \dots \times n_d$ tensors.
\end{thm}

Note that while for symmetric tensors the notions of rank and symmetric rank are not the same in general~\cite{Shitov_counterex_18}, they coincide for rank-two tensors, see, e.g.,~\cite{ZHQ_Comons_16}.

Due to the relation~\eqref{eq:distance} the theorems above are equivalent to the following statement on the maximum relative distance of a real rank-two tensor to the set of rank-one tensors.
\begin{thm}
Let $A$ be a real tensor of order $d\geq 3$ and rank at most two. Then 
\[
\min_{\rank B\leq 1}
\frac{\|A-B\|_F}{\|A\|_F}<\sqrt{ 1-\left(1-\frac{1}{d}\right)^{d-1}}
\]
 and this bound is sharp both for general as well as for symmetric tensors. Equality is achieved for the symmetric tensor $W_d$ as above.
\end{thm}

It is interesting to note that for $d\to\infty$ our results imply
\[
\min_{\substack{A \neq 0 \\ \brank(A) \le 2} } \frac{\| A \|_\sigma}{\|A\|_F} \searrow \frac{1}{\sqrt{e}} \approx 0.6065
\]
and 
\[
\max_{\substack{A \neq 0 \\ \brank(A) \le 2} }\min_{\rank B\leq 1}
\frac{\|A-B\|_F}{\|A\|_F}\nearrow \sqrt{1-\frac{1}{e}}\approx 0.7951.
\]
In particular, both quantities are bounded independently of $d$.

\subsection*{Notation}

We consider the subspace $\Sym_d(\R^n)$ of real symmetric $n \times \dots \times n$ tensors $A = [a_{i_1,\dots,i_d}]$ of order $d$. It inherits the Euclidean inner product $\langle A,B \rangle_F = \sum_{i_1,\dots,i_d} a_{i_1 \dots i_d} b_{i_1\dots i_d}$ from the ambient space, which induces the Frobenius norm via
\[
\| A \|_F^2 = \langle A, A \rangle_F.
\]
It will be convenient to introduce the notation
\[
\pm u^d = \pm u \otimes \dots \otimes u
\]
for symmetric rank-one tensors, and similarly
\[
u_1u_2\dots u_d=\frac{1}{d!}\sum_{\sigma\in\mathfrak{S}_d}u_{\sigma 1}\otimes u_{\sigma 2}\otimes \dots \otimes u_{\sigma_d}
\]
for the symmetrization of a nonsymmetric rank-one tensor $u_1 \otimes u_2 \otimes \dots \otimes u_d$. It equals the orthogonal projection of $u_1 \otimes u_2 \otimes \dots \otimes u_d$ onto $\Sym_d(\R^n)$. Specifically, the notation $u^k v^\ell$ denotes the symmetrization of the rank-one tensor $u^{\otimes k} \otimes v^{\otimes \ell}$. For symmetric rank-one tensors $u^d$ and $v^d$ it holds that $\langle u^d,v^d\rangle_F=\langle u,v\rangle^d$ and, therefore, $\|u^d\|_F=\|u\|^d$.

To any symmetric tensor $A$  one associates a homogeneous polynomial
\[
p_A(u) = \sum_{i_1,\dots,i_d} a_{i_1 \dots i_d} u_{i_1} \dots u_{i_d}=\langle A, u^d\rangle_F.
\]
The spectral norm of $A$ is then defined as
\[
\| A \|_\sigma = \max_{u\neq 0}\frac{1}{\|u\|^d}|\langle A,u^d \rangle_F|=\max_{u\neq 0} \frac{1}{\|u\|^d}|p_A(u)|.
\]
Due to a result of Banach~\cite{Banach1938}, this definition of spectral norm for symmetric tensors is consistent with the general one, which is given in~\eqref{eq: spectral norm} further below. 
If $w$ is a normalized maximizer of $\frac{1}{\|w\|^d}\abs{p_A(w)}$, then $\lambda w^d$ with $\lambda= p_A(w)=\langle A,w^d\rangle_F$ is a best symmetric rank-one approximation of $A$ in Frobenius norm, that is, it satisfies
\[
\|A-\lambda w^d\|_F=\min_{u\in\R^n,\mu\in\R}\|A-\mu u^d\|_F,
\]
and vice versa.

A symmetric tensor of rank at most two takes the form
\[
A = \alpha u^d -\beta v^d
\]
for vectors $u,v$ and scalars $\alpha,\beta\neq0 $, and the rank is equal to two if and only if $u$ and $v$ are linearly independent. Note that the difference notation will turn out to be convenient later.
 Technically, this defines tensors of \emph{symmetric rank} at most two. But since for rank two both notions of rank coincide~\cite{ZHQ_Comons_16}, we can just use the word rank throughout. It is well known that the set of tensors of rank at most two is not closed~\cite{SL_rank_08}. This is also true when restricting to symmetric tensors. The tensors in the closure are said to have border rank at most two, denoted as $\brank A  \le 2$.

\section{Proof of the main result}\label{sec: proof}

For proving Theorem~\ref{thm:main} we will determine the infimum value of the optimization problem
\begin{equation}\label{eq:optproblem}
      \inf_{\substack{\alpha, \beta\in\R\\ \|u\|=\| v\|=1}}F(\alpha,\beta,u,v) =  \frac{\|\alpha u^d - \beta v^d\|_\sigma^2}{\|\alpha u^d-\beta v^d\|_F^2}.
\end{equation}
Here we can always additionally assume that $\langle u,v\rangle \geq 0$ and $\alpha>0$. 
We will proceed in several steps. First, in section~\ref{sec:wTensor}, we validate that the tensor $W_d$, which has symmetric border rank two, achieves equality in~\eqref{eq: rank two inequality}. Hence the infimum in~\eqref{eq:optproblem} cannot be larger than $(1-\frac1d)^{d-1}$. We next consider in section~\ref{sec: optimiality condition} the first-order necessary optimality condition for~\eqref{eq:optproblem} and show that it cannot be fulfilled for rank-two tensors admitting a unique symmetric best rank-one approximation (Proposition~\ref{prop:NoCriticalPoint}). In other words, the potential candidates for achieving the infimum in~\eqref{eq:optproblem} are rank-two tensors with more than one symmetric best rank-one approximation. In section~\ref{sec:condition} we therefore derive a criterion for a symmetric rank-two tensor to have a unique symmetric best rank-one approximation (Proposition~\ref{prop: number of maximizer}), and validate by hand in sections~\ref{sec:caseSum} and~\ref{sec:caseEqual} that for tensors which do not satisfy this criterion the value of $F$ is strictly larger than $(1-\frac1d)^{d-1}$. It then remains to show in section~\ref{sec:caseLim}, that among the tensors of border rank two, and up to orthogonal transformation, only tensor $W_d$ achieves the infimum.  Taken together, these steps provide a complete proof of Theorem~\ref{thm:main}.

In our proofs we will frequently assume that $\alpha u^d - \beta v^d\in\Sym_d(\R^2)$ since we can always restrict to $\Sym_d(\spa\{u,v\})$.

\subsection{The ratio for tensor $W_d$}\label{sec:wTensor}

Recall that $W_d= e_1^{d-1}e_2^{}=\frac{d}{dt}(e_1+te_2)^d|_{t=0}$. We have $\|W_d\|_F^2=d$. The spectral norm is given by following optimization problem:
\begin{align*}
    \max\, d x^{d-1}y \quad
    \text{s.t.} \quad x^2+y^2=1.
\end{align*}
The KKT conditions for this problem lead to the relation
\begin{align*}
    (d-1)x^{d-2}y^2-x^d=0,
\end{align*}
that is, either $x=0$, or $x^2=(d-1)y^2$. We find that $\smash{x=\sqrt{\frac{d-1  }{d}}}$ and $y=\frac{1}{\sqrt{d}}$ is a maximizer with the value $\|W_d\|_\sigma=d \left(\frac{d-1}{d}\right)^{(d-1)/2}\frac{1}{\sqrt{d}}
$, and therefore 
\[
\frac{\|W_d\|^2_\sigma}{\|W_d\|_F^2}=\left(1-\frac{1}{d}\right)^{d-1}.
\]

\subsection{Optimality condition for symmetric rank-two tensors}\label{sec: optimiality condition}

The target function in~\eqref{eq:optproblem} can be written as a composition
\[
F(\alpha,\beta,u,v) = G(\varphi(\alpha,\beta,u,v))
\]
where
\[
G \colon \Sym_d(\R^n) \to \R, \quad G(A) = \frac{\| A \|_\sigma^2}{\| A \|_F^2},
\]
and
\[
\varphi \colon \R \times \R \times \R^n \times \R^n \to \Sym_d(\R^n), \quad \varphi(\alpha,\beta,u,v) = \alpha u^d - \beta v^d.
\]
While $\varphi$ is smooth, the map $G$ is not differentiable in all points. However, it is the quotient of the smooth function $A\mapsto \|A\|_F^2$ and the convex function $A\mapsto \|A\|_\sigma^2$. Therefore, the rules for generalized gradients of regular functions are applicable; see~\cite[Section 2.3]{Clarke_Book_Opt}. It follows that the subdifferential of $G$ in a point $A$ can be computed using a quotient rule, which yields
\[
\partial G(A) = \frac{2\|A\|_\sigma}{\|A\|_F^4} [ \partial(\|A\|_\sigma^{})\|A\|_F^2-A\|A\|^{}_\sigma ].
\]
Here $\partial(\|A\|_\sigma^{})$ denotes the subdifferential of the spectral norm in $A$. 
The derivative of $\varphi$ equals
\[
\varphi'(\alpha,\beta,u,v)[\delta \alpha, \delta \beta, \delta u, \delta v] = u^{d-1} (\alpha d\cdot\delta u +\delta \alpha \cdot u) - v^{d-1}(d \beta\cdot\delta v+\delta \beta \cdot v),
\]
which leads to
\begin{equation}\label{eq: subdifferential f}
\begin{aligned}
&\partial F(\alpha, \beta,u,v)[\delta \alpha, \delta \beta, \delta u, \delta v]\\
&=\frac{2\|A\|_\sigma}{\|A\|_F^4}\langle \partial(\|A\|_\sigma^{})\|A\|_F^2-A\|A\|^{}_\sigma
,u^{d-1} (\alpha d\cdot\delta u +\delta \alpha \cdot u) - v^{d-1}(d \beta\cdot\delta v+\delta \beta \cdot v ) \rangle_F
\end{aligned}
\end{equation}
with $A = \varphi(\alpha,\beta,u,v) = \alpha u^d - \beta v^d$. The subdifferential of the spectral norm can be characterized as
\begin{equation}\label{eq: subdifferential spec norm}
\partial (\|A\|_\sigma) = \conv \,\argmax \{ \langle A, X \rangle_F \colon X\in \Sym_d(\R^n),\,\rank X  = 1,\, \|X\|_F=1\}, 
\end{equation}
see~\cite[Theorem 2.1]{Clarke_Gradients_75} in general, and~\cite[Section 2.3]{AKU_Cheb_20} in particular. In words, $\partial (\|A\|_\sigma)$ equals the convex hull of the normalized symmetric best rank-one approximations of $A$.

From~\eqref{eq: subdifferential f} and~\eqref{eq: subdifferential spec norm} one concludes that the first-order optimality condition $0\in \partial F(\alpha, \beta,u,v)$ (see, e.g.,~\cite[Proposition 2.3.2]{Clarke_Book_Opt}) for problem~\eqref{eq:optproblem} implies that there exists $X$ in the convex set~\eqref{eq: subdifferential spec norm} such that
\[
\langle X - \lambda A, u^{d-1} (\alpha d\cdot\delta u +\delta \alpha \cdot u) - v^{d-1}(d \beta\cdot\delta v+\delta \beta \cdot v ) \rangle_F = 0 
\]
for all $(\delta \alpha, \delta \beta,\delta u, \delta v)$ and some $\lambda \in \R$. This is equivalent with just requiring
\[
\langle X - \lambda A, u^{d-1}\delta u+v^{d-1}\delta v \rangle_F = 0 
\]
for all $\delta u$ and $\delta v$. Let $P_{u,v}$ denote the orthogonal projection onto the linear subspace $\{u^{d-1}\delta u+v^{d-1}\delta v\colon \delta u,\delta v\in \R^n\}$ of $\Sym_d(\R^n)$. Taking into account that $P_{u,v} A = P_{u,v} (\alpha u^d-\beta v^d) = \alpha u^d-\beta v^d$, we conclude that the optimality condition can be written as
\begin{align}
\label{eq:optcondition}
    \lambda(\alpha u^d-\beta v^d)\in P_{u,v}\conv\, \argmax\{  \langle \alpha u^d- \beta v^d,X\rangle_F\colon X\in \Sym_d(\R^n),\,\rank X  = 1,\, \|X\|_F=1\}.
\end{align}

We now show that the condition~\eqref{eq:optcondition} cannot hold for tensors $\alpha u^d -\beta v^d$ admitting a unique best symmetric rank-one approximation. This is an interesting analogy to the fact that matrices achieving a minimal ratio of spectral and Frobenius norm have equal singular values.

\begin{prop}\label{prop:NoCriticalPoint}
Let $A=\alpha u^d -\beta v^d$ have rank two. If $A$
has a unique best symmetric rank-one approximation, then $A$ is not a critical point of the optimization problem~\eqref{eq:optproblem}.
\end{prop}
We use the following lemma that shows
$ P_{u,v} w^d=au^{d-1}w+bv^{d-1}w$ for any $w\in\R^n$ with some $a,b\in\R$.
\begin{lm}\label{lm:projection}
Let $\|u\|=\|v\|=1$. The projection $ P_{u,v}w^d$ is given by
\begin{align*}
   \frac{1}{1-\langle u,v \rangle^{2d-2}}\big[(\langle u,w \rangle^{d-1}-\langle u,v \rangle^{d-1}\langle v,w \rangle^{d-1})u^{d-1}
   +(\langle v,w \rangle^{d-1}-\langle u,v \rangle^{d-1}\langle u,w \rangle^{d-1})v^{d-1}\big]w.
\end{align*}

\end{lm}
\begin{proof}
This follows from the definition of orthogonal projection by a direct calculation.
\end{proof}

\begin{proof}[Proof of Proposition~\ref{prop:NoCriticalPoint}]
Let one of $\pm w^d$ be the normalized best symmetric rank-one approximation of~$A$. Since it is unique, the optimality condition becomes
\begin{equation}\label{eq: optimality new}
\lambda(\alpha u^d+\beta v^d)\in  P_{u,v} w^d.
\end{equation}
From $p_A(w)=\langle A, w^d\rangle_F\neq 0$ and
\(
A = \alpha u^d +\beta v^d\in \{u^{d-1}\delta u+v^{d-1}\delta v\colon \delta u,\delta v\in \R^n\}
\) we have $ P_{u,v}w^d\neq 0$, which excludes $\lambda = 0$. 
By Lemma~\ref{lm:projection}, $ P_{u,v}w^d= au^{d-1}w+bv^{d-1} w$ for some $a,b\in\R$. However, since $u$ and $v$ are linearly independent, we have the decomposition
\[
\{u^{d-1}\delta u+v^{d-1}\delta v\colon \delta u,\delta v\in \R^n\}=\{u^{d-1}\delta u\colon \delta u\in \R^n\}\oplus \{v^{d-1}\delta v\colon \delta v\in \R^n\}
\]
into two complementary subspaces.
Therefore,~\eqref{eq: optimality new} would only be possible if $w$ is both a multiple of $u$ and $v$, which contradicts the linear independence of $u$ and $v$.
\end{proof}

\subsection{A condition for unique symmetric best rank-one approximation}\label{sec:condition}

We now present a class of symmetric rank-two tensors admitting unique best symmetric rank-one approximations. By the result of Proposition~\ref{prop:NoCriticalPoint} these can then be excluded from the further discussion on the minimal norm ratio.

\begin{prop}\label{prop: number of maximizer}
Let
\[
A=\alpha u^d-\beta v^d
\]
with $u\neq v$, $\|u\|=\|v\|=1$, $\langle u,v\rangle \geq 0$ and $\alpha> \beta>0$. Then $A$ has exactly
one best symmetric rank-one approximation. 
\end{prop}

For the proof we require auxiliary results. One is the following fact about polynomials.
\begin{lm}\label{lm: number of critical p}
Let $a, \gamma >0$ and $b \geq0$ and $d\geq 2$. The equation $x=\gamma (x-a)(x+b)^{d-1}$ has two real solutions if $d$ is even, and three real solutions if $d$ is odd.
\end{lm}

\begin{proof}
Let $p(x)=\gamma (x-a)(x+b)^{d-1}-x$. Then by the intermediate value theorem, $p$ must have at least two real zeros, namely one in the interval $[-b,0]$ and another one in the interval $(a,\infty)$. 
 On the other hand,  
 \[
 p'(x)= \gamma d(x+b)^{d-2}\left(x-\frac{(d-1)a-b}{d}\right)-1,
 \]
 has at most two sign changes, one at a value larger than $\frac{(d-1)a-b}{d}$ and another at one at a value smaller than $-b$ if $d$ is odd. Therefore, $p$ has at most three real zeros. 
The statement follows from the fact that the number of real zeros of a polynomial with real coefficients has the same parity as its degree.
\end{proof}

The second lemma narrows the possible locations of maximizers of the homogeneous form~$\abs{p_A}$.
 
\begin{lm}\label{lm:position_maximizer}
Under the assumptions of Proposition~\ref{prop: number of maximizer}, let $w$ be a maximizer of $\abs{p_A(w)} =  \abs{\langle\alpha u^d-\beta v^d,w^d \rangle_F}$ subject to $\|w\|=c>0$. Then $\abs {\langle u,w\rangle} \geq \abs{\langle v,w\rangle}$.
\end{lm}
\begin{proof}
Assume to the opposite that $\abs{\langle u,w\rangle} < \abs{\langle v,w\rangle}$ and without loss of generality $\langle v,w\rangle>0$. Let $Q$ be the symmetric orthogonal matrix mapping $u$ to $v$ and $v$ to $u$ (i.e. $Q = I - z z^T$ with $z = (u+v)/\| u+v\|$), and let $\bar{w}=Qw$. Then $\langle u,w\rangle=\langle v,\bar{w}\rangle$ and $\langle v,w\rangle=\langle u,\bar{w}\rangle$. By assumption, we then have
\[
 \abs{\langle\alpha u^d-\beta v^d,\bar{w}^d \rangle_F}=\langle\alpha u^d-\beta v^d,\bar{w}^d \rangle_F.
\]
If $\abs{\langle\alpha u^d-\beta v^d,{w}^d \rangle_F}=\langle\alpha u^d-\beta v^d,{w}^d\rangle_F$ this yields $\abs{\langle\alpha u^d-\beta v^d,\bar{w}^d \rangle_F}>\abs{\langle\alpha u^d-\beta v^d,{w}^d \rangle_F}$ (by using $(\alpha+\beta)\langle v,w\rangle^d>(\alpha+\beta)\langle u, w\rangle ^d$) which contradicts the optimality of $w$. In the other case, $\abs{\langle\alpha u^d-\beta v^d,{w}^d \rangle_F}= - \langle\alpha u^d-\beta v^d,{w}^d\rangle_F$, optimality implies $\beta(\langle u,w\rangle^d+\langle v,w\rangle^d)>\alpha(\langle u,w\rangle^d+\langle v,w\rangle^d)$ which contradicts $\alpha > \beta$.
\end{proof}

We are now in the position to prove Proposition~\ref{prop: number of maximizer}.

\begin{proof}[Proof of Proposition~\ref{prop: number of maximizer}]
We can assume that $A \in \Sym_d(\R^2)$, so that $u,v \in \R^2$.
Without loss of generality, since we can change coordinates, we can consider $\alpha=1$, $u=\begin{pmatrix}0\\1 \end{pmatrix}$ and $\sqrt[d]{\beta}v=\begin{pmatrix}a\\b\end{pmatrix}$ with $a>0$, $b\geq0$ (since $\langle u,v\rangle \geq 0$),  and $a^2+b^2<1$ (since $\beta<\alpha=1$). Writing $w = \lambda \begin{pmatrix}x\\y\end{pmatrix}$ for points on the unit circle, where $\lambda > 0$ is a normalization constant, we then have
\begin{equation}\label{eq: p_Aw}
p_A(w) = \lambda^d[y^d - (ax + b y)^d].
\end{equation}
Critical points on the circle are characterized by $\langle w, \nabla p_A(w) \rangle = 0$, which means
\[
y^{d-1}x-(bx-ay)(ax+by)^{d-1} = 0
\]
independent of $\lambda$. Note that here $y=0$ is not possible since both $a$ and $b$ are nonzero. 
Recall that a symmetric best rank-one approximation of $A$ is given as $p_A(w)w^d$, where $w$ maximizes $\abs{p_A(w)}$ on the circle.
Since $p_A(-w)=(-1)^dp_A(w)$, in order to prove the assertion it suffices to show that $\abs{p_A(w)}$ has exactly one maximizer $w$ with $y=1$.
The optimality condition at such a $w$ reduces to
\begin{equation}\label{eq:critpoint}
x = (bx-a)(ax+b)^{d-1}.
\end{equation}
Hence, we only need to show that there is exactly one solution $x$ of this equation corresponding to a global maximum of $\abs{p_A}$ on the unit circle.

 If $y=1$, then $p_A$ in~\eqref{eq: p_Aw} has a zero at $x_0=\frac{1-b}{a}$. Then
\[
x_0=\frac{1-b}{a}>\frac{b-b^2-a^2}{a}= (bx_0-a)(ax_0+b)^{d-1}.
\]
This shows that~\eqref{eq:critpoint} has at least one solution $x^*>x_0$. We consider such a solution $x^*$ such that the corresponding unit vector $w=\lambda\begin{pmatrix}
x^*\\1
\end{pmatrix}$ is a local maximum of $\abs{p_A}$ on the unit circle. We have
\[
\abs{\langle u,w\rangle}=\lambda <\frac{\lambda}{\sqrt[d]{\beta}}=  \frac{\lambda}{\sqrt[d]{\beta}}(ax_0-b) <\lambda \frac{1}{\sqrt[d]{\beta}}(ax^*-b) =\abs{\langle v,w\rangle}.
\]
By Lemma~\ref{lm:position_maximizer}, $w$ is not a global maximum of $\abs{p_A}$. If $d$ is even, then by Lemma~\ref{lm: number of critical p} equation~\eqref{eq:critpoint} has exactly two solutions and therefore only one corresponds to a global maximum. If $d$ is odd, then by the same lemma~\eqref{eq:critpoint} has three solutions. Taking into account that $p_A$ in~\eqref{eq: p_Aw} has only one zero for $y=1$, one of these solutions corresponds to a local minimizer of $\abs{p_A}$. Hence, there is only one global maximizer.
\end{proof}

\subsection{The case $\alpha> 0\geq \beta$}\label{sec:caseSum}
We show that $\frac{\|\alpha u^d-\beta v^d\|_\sigma^2}{\|\alpha u^d-\beta v^d\|_F^2} \geq\frac{1}{2}$ if $\langle u^d,v^d\rangle_F\geq 0$ and $\alpha > 0 \ge \beta$. This shows that for $d> 2$ such tensors do not attain the infimum in~\eqref{eq:optproblem} since $\frac{1}{2}> \left(1-\frac{1}{d}\right)^{d-1} $. We formulate this statement without $\alpha$ and $\beta$ by removing the restriction $\| u \| = \| v \| = 1$.

\begin{prop}\label{lm:sum}
Let $u \neq v$ and $\langle u,v \rangle\geq0$. Then $\frac{\|u^d+v^d\|^2_\sigma}{\|u^d+v^d\|^2_F}\geq \frac{1}{2}$.
\end{prop}
\begin{proof}
We can assume $\|u\|\geq\|v\|$. Using that  $\|u^d+v^d\|_\sigma\geq \langle u^d+v^d,\frac{u^d}{\|u\|^d}\rangle_F$, we have
\begin{align*}
    \frac{\|u^d+v^d\|^2_\sigma}{\|u^d+v^d\|^2_F}
    \geq \frac{\|u\|^{2d}+2\langle u,v\rangle^d+\left(\frac{\langle u,v\rangle }{\|u\|}\right)^{2d}}
    {\|u\|^{2d}+2\langle u,v\rangle^d+\|v\|^{2d}}
    =1-\frac{\|v\|^{2d}-\left(\frac{\langle u,v\rangle }{\|u\|}\right)^{2d}} 
  {\|u\|^{2d}+2\langle u,v\rangle^d+\|v\|^{2d}}
  &\geq
  1-\frac{\|v\|^{2d}}{2\|v\|^{2d}}=\frac{1}{2},
\end{align*}
as asserted.
\end{proof}

\subsection{The case $\alpha=\beta>0$}\label{sec:caseEqual}

In this section we verify by a direct calculation that the infimum in~\eqref{eq:optproblem} is not attained for the difference of two rank-one tensors with the same norm, i.e. when $\alpha= \beta$ in~\eqref{eq:optproblem}.
\begin{prop}\label{prop:diffequal}
Let $u\neq v$, $\|u\|=\|v\|\neq 0$, $\langle u,v\rangle\geq 0$ and $d\geq 3$. Then 
\[
\frac{\|u^d-v^d\|_\sigma^2}{\|u^d-v^d\|^2_F}> \left(1-\frac{1}{d}\right)^{d-1}.
\]
\end{prop}

We require the following version of Jensen's inequality.
\begin{lm}\label{lm:conv}
Let $f:[a,b]\to \R$ be convex and continuously differentiable. If $a+b=a'+b'$ and $a<a'<b'<b$, then
\[
\frac{1}{b-a}\int_a^bf(x)\, dx\geq \frac{1}{b'-a'}\int_{a'}^{b'} f(x) \, dx\geq f\left(\frac{a+b}{2}\right).
\]
The inequalities are strict if $f$ is strictly convex.
\end{lm}
\begin{proof}
Without loss of generality let $a=-b$ and $a'=-b'$. Then using substitution we have
\begin{align*}
   \frac{1}{b}\int_{-b}^bf(x) \, dx&= \frac{1}{b'}\int_{-b'}^{b'} f\left(\frac{b}{b'}x\right)-f(x)+f(x)\,dx\\
   &=\frac{1}{b'}\int_{-b'}^{b'}f(x) \, dx+\frac{1}{b'}\int_0^{b'}\int_x^{\frac{bx}{b'}}f'(y)-f'(-y)\,dydx \geq\frac{1}{d}\int_{-b'}^{b'}f(x) \, dx,
\end{align*}
by monotonicity of the derivative of a convex function. This shows the first of the asserted inequalities. The second inequality is just Jensen's inequality, noting that $\frac{a+b}{2}=\frac{a'+b'}{2}$. If $f$ is strictly convex, then $f'$ is strictly monotone and the inequalities are strict.
\end{proof}

\begin{proof}[Proof of Proposition~\ref{prop:diffequal}]
We can assume that $A \in \Sym_d(\R^2)$, so that $u,v \in \R^2$.
After rotation and rescaling we have $u=\begin{pmatrix}1\\t\end{pmatrix}$ and $v=\begin{pmatrix}1\\-t\end{pmatrix}$ with $t\in (0,1].$ Then
\begin{equation}\label{eq:frobnormofud-vd}
\|u^d-v^d\|^2_F=2(1+t^2)^d-2(1-t^2)^d\eqqcolon g(t).
\end{equation}
First, we apply the estimate
\[
\|u^d-v^d\|_\sigma \geq \left\langle u^d-v^d,\frac{u^d}{\|u\|^d}\right\rangle_F=\frac{(1+t^2)^d-(1-t^2)^d}{\sqrt{1+t^2}^d},
\]
which yields
\[
\frac{\|u^d-v^d\|_\sigma^2}{\|u^d-v^d\|_F^2}\geq\frac{(1+t^2)^d-(1-t^2)^d}{2(1+t^2)^d}=\frac{1}{2}\left( 1-\left(\frac{1-t^2}{1+t^2}\right)^d\right).
\]
The right-hand side is monotonically increasing in the interval $(0,1]$. For $t=\sqrt{\frac{1}{d-1}}$ it equals
\[
\frac{1}{2}\left( 1-\left(\frac{d-2}{d}\right)^d\right)=\frac{d^d-(d-2)^d}{2d^d}.
\]
This value is larger than $\left( 1-\frac{1}{d}\right)^{d-1}=\left( \frac{d-1}{d}\right)^{d-1}$ since, using Lemma~\ref{lm:conv} with $f(t)=t^{d-1}$, it holds that
$d^d-(d-2)^d> 2 d (d-1)^{d-1}$
for $d\geq 3$. This shows that
\[
\frac{\|u^d-v^d\|_\sigma^2}{\|u^d-v^d\|_F^2}>\left(1-\frac{1}{d}\right)^{d-1}
\]
for all $t\in \left[\sqrt{\frac{1}{d-1}},1\right]$.
It hence remains to verify this inequality 
 for all $t\in \left( 0,\sqrt{\frac{1}{d-1}}\right)$, which is a little bit more involved. The starting point is another lower bound for the spectral norm, namely
\[
\|u^d-v^d\|_\sigma \geq \left\langle u^d-v^d, \begin{pmatrix}\sqrt{(d-1)/d}\\1/\sqrt{d}\end{pmatrix}^d\right\rangle_F=\frac{1}{\sqrt{d}^d}\Bigg(\left( \sqrt{d-1}+t\right)^d-\left(\sqrt{d-1}-t\right)^d\Bigg)\eqqcolon h(t).
\]
Note that $\frac{u^d-v^d}{\|u^d-v^d\|_F}\to\frac{W_d}{\|W_d\|_F} $ for $t\to 0$. This can be seen by taking the limit of $\frac{u^d-v^d}{t}$ and noting that $g(t)=\|u^d-v^d\|^2_F$ is of order $t^2$ by~\eqref{eq:frobnormofud-vd}. We therefore have
\[
\lim_{t\to0}\frac{h(t)^2}{g(t)}=\left\langle \frac{W_d}{\|W_d\|_F},\begin{pmatrix}\sqrt{(d-1)/d}\\1/\sqrt{d}\end{pmatrix}^d\right\rangle_F^2=\frac{\|W_d\|_\sigma^2}{\|W_d\|^2_F}=\left(1 - \frac{1}{d}\right)^{d-1},
\]
where the second and third equalities have been shown in section~\ref{sec:wTensor}. We now claim that 
\[
\frac{d}{dt}\frac{h(t)^2}{g(t)}>0\text{ for $\textstyle t\in \left( 0,\sqrt{\frac{1}{d-1}}\right)$}
\]
which then proves the assertion. This claim is equivalent to the positivity of  
\begin{align*}
\frac{\sqrt{d}^d}{4d}\big(2h'(t)g(t)-g'(t)h(t)\big)
&=
\Bigg[ \! \left(\sqrt{d-1}+t\right)^{d-1}+\left(\sqrt{d-1}-t\right)^{d-1}\Bigg]\Bigg[(1+t^2)^d-(1-t^2)^d\Bigg]\\
&\quad-t\Bigg[\! \left(\sqrt{d-1}+t\right)^{d}-\left(\sqrt{d-1}-t\right)^{d}\Bigg]
\Bigg[(1+t^2)^{d-1}+(1-t^2)^{d-1}\Bigg].
\end{align*}
Elementary manipulations give
\begin{align}\label{eq:rewrite}
&\frac{\sqrt{d}^d}{4d}\big(2h'(t)g(t)-g'(t)h(t)\big)\notag
\\
&=
   \Bigg[\!\left(\sqrt{d-1}+t\right)^{d-1} (1+t^2)^{d-1}-\left(\sqrt{d-1}-t\right)^{d-1} (1-t^2)^{d-1}\Bigg]
   \left(1-t\sqrt{d-1}\right)\notag
  \\
&\quad-\Bigg[\!\left(\sqrt{d-1}+t\right)^{d-1} (1-t^2)^{d-1}-\left(\sqrt{d-1}-t\right)^{d-1} (1+t^2)^{d-1}\Bigg]
\left(1+t\sqrt{d-1}\right)\notag
\\
&\begin{aligned}
&=\Bigg[\Big(\underbrace{\sqrt{d-1}+t+t^2\sqrt{d-1}+t^3}_{{}\eqqcolon{} b}\Big)^{d-1}-\Big(\underbrace{\sqrt{d-1}-t-t^2\sqrt{d-1}+t^3}_{{}\eqqcolon{} a}\Big)^{d-1}\Bigg]
\left(1-t\sqrt{d-1}\right)
\\
&\quad-\Bigg[\Big(\underbrace{\sqrt{d-1}+t-t^2\sqrt{d-1}-t^3}_{\eqqcolon b'}\Big)^{d-1}-\Big(\underbrace{\sqrt{d-1}-t+t^2\sqrt{d-1}-t^3}_{\eqqcolon a'}\Big)^{d-1}\Bigg]
\left(1+t\sqrt{d-1}\right).
\end{aligned}
\end{align}
Note that for $t\in\left(0,\sqrt{\frac{1}{d-1}}\right)$ we have $b> b' > a' > a$
and 
\[
    b-a=2t\left(1+t\sqrt{d-1}\right),\quad   b'-a'   =2t\left(1-t\sqrt{d-1}\right).
\]
Therefore with $f(t)=(d-1) t^{d-2}$, we can  rewrite~\eqref{eq:rewrite} as 
\[
\frac{1}{4d}\sqrt{d}^d\big(2h'(t)g(t)-g'(t)h(t)\big)=\frac{1}{2t}\left[(b'-a')\int_a^b f(x) \, dx -(b-a) \int_{a'}^{b'} f(x) \, dx \right].
\]
Moreover,
\begin{equation*}
\frac{a+b}{2} =\sqrt{d-1}+2t^3    
    >
   \sqrt{d-1}-2t^3=\frac{a'+b'}{2},
\end{equation*}
and therefore  $a''\coloneqq \frac{a+b -(b'-a')}{2}>a'>a$ and $b>b''\coloneqq \frac{a+b+(b'-a')}{2}>b'$. Since ${a''+b''}={a+b}$ and $a''-b''=a'-b'$, Lemma~\ref{lm:conv} yields
\[
(b'-a')\int_a^b f(x) \, dx \ge(b-a) \int_{a''}^{b''} f(x) \, dx >(b-a) \int_{a'}^{b'} f(x) \, dx,
\] 
where the second inequality follows from monotonicity of $f$. This shows that~\eqref{eq:rewrite} is positive.
\end{proof}

\subsection{Tensors of border rank two}\label{sec:caseLim}
We now consider tensors lying on the boundary of the set of symmetric rank-two tensors.

\begin{prop}\label{prop:boundary}
Let $A$ be a limit of symmetric rank-two tensors and $\rank A>2$. Then
\[
\frac{\|A\|_\sigma^2}{\|A\|^2_F}\geq \left(1-\frac{1}{d}\right)^{d-1}\!=\frac{\|W_d\|^2_\sigma}{\|W_d\|^2_F}
\]
and equality is attained if and only if $A=u^{d-1}v$ for some orthogonal $u$ and $v$, that is, for tensors arising from scaling and orthogonal transformations of tensor $W_d$.
\end{prop}
The boundary of rank-two tensors is well studied. We require the following well-known parametrization, see, e.g.,~\cite{BL_third_14}. We offer a self-contained proof for completeness.
\begin{lm}\label{lm:boundarytensors}
Let $A$ be a limit of symmetric rank-two tensors and $\rank A>2$. Then $A$ is of the form 
\[
A
=a u^d+bdu^{d-1}v
\] 
with $\langle u,v\rangle=0$ and $\|u\|=\|v\|=1$. 
\end{lm}
\begin{proof}
Let $A_n=u_n^d\pm v_n^d$ with $\lim_{n\to\infty}A_n=A$ or $\lim_{n\to\infty}A_n=-A$. It is not difficult to see that $u_n$ and $v_n$ must be unbounded since otherwise there is a subsequence of $A_n$ converging to a tensor of rank at most two, contradicting $\rank A>2$.
We write $v_n=s_n u_n+t_n w_n$ with $\|w_n\|=1$ and $\langle u_n, w_n\rangle=0$. Then 
\[
A_n=(1\pm s_n^d)u_n^d\pm\sum_{k=1}^d \binom{d}{k} s_n^{d-k} t_n^k  u_n^{d-k}w_n^k,
\] and it can be checked that all terms are pairwise orthogonal. Hence, since $A_n$ converges, all terms must be bounded and by passing to a subsequence we can assume that all of them converge. Due to  $\|u_n\|\to\infty$ we have $1\pm s_n^d\to 0$ for the first term, which implies that the sequence $s_n$ is bounded. Therefore, considering the term $k=1$,  the sequence $t_n\|u_n\|^{d-1}$ is bounded which automatically implies $t_n^k\|u_n\|^{d-k}\to 0$ for all $k>1$.  We conclude that
\[
\lim_{n\to\infty} A_n=\lim_{n\to\infty} ( 1\pm s_n^d) u_n^d+\lim_{n\to\infty}ds_n^{d-1}t_n u_n^{d-1}w_n = a u^d+bd u^{d-1} v
\]
which proves the assertion.
\end{proof}
\begin{proof}[Proof of Proposition~\ref{prop:boundary}]
Using Lemma~\ref{lm:boundarytensors}, scaling and orthogonal transformations, we can assume  
$A=a e_1 ^d+bd e_1^{d-1}e_2 \in\Sym_d(\R^2)$ with $a,b\geq0$. Then $\|A\|_F^2=a^2+b^2d$ since the tensors $e_1^d$ and $e_1^{d-1}e_2$ are orthogonal and $\|de_1^{d-1}e_2\|_F^2=d$.
We have the following two lower bounds  for the spectral norm:
\begin{equation}\label{eq:1lowerbound}
\|A\|_\sigma\geq\left \langle ae_1^d+bde_1^{d-1}e_2 ,\frac{1}{\sqrt{d}^d}\begin{pmatrix}\sqrt{d-1}  \\ 1\end{pmatrix}^d  \right\rangle_F =\frac{1}{\sqrt{d}^d}\left(a\sqrt{d-1}^d+bd\sqrt{d-1}^{d-1}\right)
\end{equation}
and
\begin{equation}\label{eq:2lowerbound}
\|A\|_\sigma\geq \left \langle ae_1^d+bde_1^{d-1}e_2 ,e_1^d\right\rangle _F =
a.
\end{equation}
We can restrict to tensors $A$ with Frobenius norm $\|A\|^2_F=a^2+b^2d=1$ and need to show that 
\[
\|A\|_\sigma>\left(1-\frac{1}{d}\right)^{\frac{d-1}{2}}
\]
whenever $a>0$. The first lower bound~\eqref{eq:1lowerbound}  implies that this is true whenever $b> \frac{\sqrt{d}-a\sqrt{d-1}}{d}$. Together with $1=a^2+b^2d$ and $a,b\geq0$ this verifies the claim for $0< a< \frac{2\sqrt{d(d-1)}}{2d-1}$. If $a\geq  \frac{2\sqrt{d(d-1)}}{2d-1}$, then the  second lower bound~\eqref{eq:2lowerbound} yields the desired estimate
\[
\|A\|_\sigma^2\geq a^2\geq\left(\frac{2\sqrt{d(d-1)}}{2d-1}\right)^2>\frac{d-1}{d}>\left(1-\frac{1}{d}\right)^{d-1}
\]
for $d\geq 3$.
\end{proof}
This concludes the proof of Theorem~\ref{thm:main}.

\section{Approximation ratio for nonsymmetric rank-two tensors}\label{sec: nonsymmetric}

Recall that the spectral norm for general $n_1 \times \dots \times n_d$ tensors is defined as
\begin{equation}\label{eq: spectral norm}
\|A \|_\sigma = \max_{\| u_1 \| = \dots = \| u_d \| = 1 }\langle A, u_1 \otimes \dots \otimes u_d \rangle_F.
\end{equation}
The result for symmetric tensors raises the question whether
the inequality
\[
\|A\|_\sigma>\left(1-\frac{1}{d}\right)^{\frac{d-1}{2}}\|A\|_F 
\]
is also true for general real tensors of order $d \ge 3$ and rank at most two. 
As stated in Theorem~\ref{thm:main2}, the answer is indeed affirmative and a consequence of the following interesting fact.

\begin{prop}\label{prop:makesymmetric}
Let $A$ be a real $n_1 \times \dots \times n_d$ tensor of rank at most two. Then there is is symmetric rank-two tensor $A_\mathsf S\in\Sym_d (\R^2)$ with $\|A\|_F=\| A_\mathsf S\|_F$ and $\|A\|_\sigma\geq\| A_\mathsf S\|_\sigma$.
\end{prop}

For the proof, we will require two lemmas. The first is on the behavior of successively taking geometric means of positive real numbers, and the second on the relation of Frobenius and spectral norm of two particular $2\times 2$ matrices.

\begin{lm}\label{lm:geometricmeans}
Let $x,z\geq 0$, $k>0$, and define the sequence
\[
y_0=x,\quad y_1=\left(x^{k-1}z\right)^{\frac{1}{k}}, \quad y_{\ell+2}=\left(y_{\ell+1}^{k-1}y_{\ell}^{}\right)^{\frac{1}{k}}.
\]
Then $\lim_{\ell\to\infty} y_\ell=\left(x^k z\right)^{\frac{1}{k+1}}$.
\end{lm}
\begin{proof}
We may assume  $x,z> 0$, otherwise the result follows immediately. We show via induction that
\begin{equation}\label{eq:inductionlmgeom}
y_\ell=\left(x^{{k^{\ell+1}+(-1)^{\ell}}}z^{{k^{\ell}+(-1)^{\ell-1}}}\right)^{\frac{1}{k^{\ell}(k+1)}}.
\end{equation}
The cases $\ell=0$ and $\ell=1$ follow directly. Now let \eqref{eq:inductionlmgeom} be true for $1,\ldots,\ell+1$. Then 
\begin{align*}
y_{\ell+2}
&=\left(y_{\ell+1}^{k-1}y_\ell^{}\right)^{\frac{1}{k}}
=x^{\left(\frac{(k-1)\left(k^{\ell+2}+(-1)^{\ell+1}\right)}{k^{\ell+2}(k+1)}+\frac{k^{\ell+1}+(-1)^{\ell}}{k^{\ell+1}(k+1)}\right)}
z^{\left(\frac{(k-1)\left(k^{\ell+1}+(-1)^\ell\right)}{k^{\ell+1}(k+1)}+\frac{k^\ell+(-1)^{\ell-1}}{k^\ell(k+1)}\right)}
\\
&=\left(x^{{k^{\ell+3}+(-1)^{\ell+2}}}  z^{{k^{\ell+2}+(-1)^{\ell+1}}}\right)^{\frac{1}{k^{\ell+2}(k+1)}},
\end{align*}
proving \eqref{eq:inductionlmgeom}. Taking the limit $\ell\to\infty$ gives the result.
\end{proof}

\begin{lm}\label{lm:matrixspectralnorm}
Let $a,b \in \R$ and $0\leq x_1,x_2\leq 1$. Define the matrices 
\[
S=\begin{pmatrix}
a+bx_1x_2 &b\sqrt{x_1x_2-x_1^2x_2^2} \\
b\sqrt{x_1x_2-x_1^2x_2^2} & b(1-x_1x_2)
\end{pmatrix}
\quad\! \text{and}\quad\!
T=\begin{pmatrix}
a+bx_1x_2 &bx_1\sqrt{1-x_2^2} \\
bx_2\sqrt{1-x_1^2} & b\sqrt{(1-x_1^2)(1-x_2^2)}
\end{pmatrix}.
\]
Then $\|S\|_F=\|T\|_F$ and $\|S\|_\sigma\leq\|T\|_\sigma $.
\end{lm}
\begin{proof}
A direct calculation shows that $\|S\|_F=\|T\|_F$. The singular values of $2\times 2$ matrices are given by $\sigma_{1,2}^2={F^2}/{2}\pm\sqrt{{F^4}/{4}-\abs{D}^2}$, where $F$ is the Frobenius norm and $D$ is the determinant of the matrix. We have
\[
\abs{\det S}^2=a^2b^2(1-2x_1x_2+x_1^2x_2^2) \quad \text{and}\quad 
\abs{\det T}^2=a^2b^2(1-x_1^2-x_2^2+x_1^2x_2^2).
\]
Since $2x_1x_2\leq x_1^2+x_2^2$ implies $\abs{\det T}^2 \leq\abs{\det S}^2$, the largest singular value of $T$, which equals its spectral norm, is larger or equal to the largest singular value of $S$.
\end{proof}

\begin{proof}[Proof of~\Cref{prop:makesymmetric}]
Write $A=\alpha U + \beta V$ where $U=u_1 \otimes \dots \otimes u_d$ and $V=v_1 \otimes \dots \otimes v_d$ with $\|u_i\|=\|v_i\|=1$. Then $\|A\|_F^2=\alpha^2 + 2\alpha\beta\langle U, V\rangle_F+\beta^2$. We may assume that $u_i,v_i\in\R^2$ and after an orthogonal change of bases and possibly changing sign of $\beta$, we may also assume that
\[
u_i=e_1, \quad v^{}_i=x_i^{}e^{}_1+ \sqrt{1-x_i^2}e^{}_2 \quad \text{with \quad $0\leq x_i\leq 1$.}
\]
Our goal is to show that replacing any $k$ factors $v_{i_1},\ldots,v_{i_k}$ of $V$ with the same unit norm vector $v$ defined by
\[
v=x e_1 +\sqrt{1-x^2} e_2 \quad \text{with} \quad x= \Bigg(\prod_{j=1}^k x_{i_j}\Bigg)^{1/k}
\]
leads to a tensor with the same Frobenius norm but smaller spectral norm. Since Frobenius and spectral norm are invariant under permutation of tensor factors, it suffices to prove this for the case that the first $k$ vectors $v_1,\dots,v_k$ are replaced in this way. The resulting tensor is denoted by $A_k=\alpha U+\beta V_k$ with
$
V_k= v \otimes \dots \otimes v \otimes v_{k+1} \otimes \dots \otimes v_d
$ and since \[
\langle U, V_k\rangle_F=\prod_{i=1}^k \langle u_i,v\rangle\prod_{i=k+1}^d \langle u_i,v_i \rangle=x^k \prod_{i=k+1}^d x_i=\prod_{i=1}^d x_i=\prod_{i=1}^d \langle u_i,v_i \rangle=\langle U, V\rangle_F,
\]
the Frobenius norms of $A$ and $A_k$ indeed coincide. In the remainder of the proof we show by induction that the spectral norm does not increase with $k$, i.e., $\|A_{k+1}\|_\sigma \leq \|A_{k}\|_\sigma\leq \|A\|_\sigma$. For $k=d$ this provides a symmetric tensor with the desired properties.

We start with $k=2$. Let $w_1,\ldots, w_d$ be the maximizers in 
\[
\max_{\|w_1\|=\cdots=\|w_d\|=1}\langle A_2,w_1 \otimes \dots \otimes w_d\rangle_F=\|A_2\|_\sigma.
\]
Let $a=\alpha \prod_{i=3}^d\langle u_i,w_i \rangle$, $b=\beta \prod_{i=3}^d\langle v_i,w_i \rangle$, and consider the matrices
\[
T=a e_1^{} e_1^{T}+ b v_1^{}v_2^{T}=\begin{pmatrix}
a+bx_1x_2 &bx_1\sqrt{1-x_2^2} \\
bx_2\sqrt{1-x_1^2} & b\sqrt{(1-x_1^2)(1-x_2^2)}
\end{pmatrix}
\]
and
\[
S=a e_1^{} e_1^{T}+ b v^{}v_{}^{T} = \begin{pmatrix}
a+bx_1x_2 &b\sqrt{x_1x_2-x_1^2x_2^2} \\
b\sqrt{x_1x_2-x_1^2x_2^2} & b(1-x_1x_2)
\end{pmatrix}.
\]
They represent the bilinear forms
\[
\tilde w_1^T T \tilde w_2=\langle A, \tilde w_1\otimes \tilde w_2 \otimes w_3 \otimes \dots \otimes w_d \rangle_F\quad \text{and}\quad \tilde w_1^T S \tilde w_2=\langle A_2, \tilde w_1\otimes \tilde w_2 \otimes w_3 \otimes \dots \otimes w_d \rangle_F
\]
in $\tilde w_1 $ and $ \tilde w_2$. Clearly,
\[
\|T\|_\sigma=
\max_{\|\tilde w_1\|=\|\tilde w_2\|=1}\langle A, \tilde w_1\otimes \tilde w_2 \otimes w_3 \otimes \dots \otimes w_d \rangle_F\leq \|A\|_\sigma
\]
and 
\[
\|S\|_\sigma=
\max_{\|\tilde w_1\|=\|\tilde w_2\|=1}\langle A_2, \tilde w_1\otimes \tilde w_2 \otimes w_3 \otimes \dots \otimes w_d \rangle_F
=
\|A_2\|_\sigma.
\]\Cref{lm:matrixspectralnorm} implies $\|S\|_\sigma\leq \|T\|_\sigma$ and therefore $\|A_2\|_\sigma\leq \|A\|_\sigma$.

For the induction step, let $2 \le k < d$ and assume that replacing any $k$ factors of $V$ in the described manner always results in a tensor with a smaller or equal spectral norm. Note that here $V$ was in principle arbitrary. Starting from the given $V$, we now construct a sequence $\widetilde V_0, \widetilde V_1, \dots$ of rank-one tensors in which the first $k$ factors and then the second to $(k+1)$-st factors are successively replaced:
\begin{align*}
\widetilde{V}_0 &= \tilde v_0\otimes \dots \otimes \tilde v_0 \otimes v_{k+1} \otimes (v_{k+2} \otimes \dots \otimes v_d), \\
\widetilde{V}_1 &= \tilde v_0 \otimes \tilde v_1 \otimes \dots \otimes \tilde v_{1 \phantom{{}+k}} \otimes (v_{k+2} \otimes \dots \otimes v_d), \\
\widetilde{V}_2 &= \tilde v_2\otimes \dots \otimes \tilde v_2 \otimes \tilde v_{1 \phantom{{}+k}} \otimes (v_{k+2} \otimes \dots \otimes v_d) \\
&\vdots
\end{align*}
and so on (the term in brackets disappears when $k=d-1$). By induction hypothesis, the corresponding sequence $B_\ell=\alpha U+\beta\widetilde{V}_\ell$ of tensors has nonincreasing spectral norm and in particular $\| B_\ell \|_\sigma \le \| B_0 \|_\sigma = \| A_k \|_\sigma \le \| A \|_\sigma$. We claim the $B_\ell$ converge to $A_{k+1}$, which proves $\| A_{k+1} \|_\sigma \le \| A_k \|_\sigma \le \| A \|_\sigma$ as desired. Indeed, the unit norm vectors
\[
\tilde v^{}_\ell=y^{}_\ell e^{}_1+\sqrt{1-y_\ell^2}e_2^{}
\]
above are constructed according to 
\[
y_0=x =\prod_{i=1}^k x_i^{1/k},\quad y_1= \left(x^{k-1} x_{k+1}\right)^\frac{1}{k},\quad y_{\ell+2}
=\left(y_{\ell+1}^{k-1} y^{}_{\ell}\right)^\frac{1}{k}.
\]
 By \Cref{lm:geometricmeans}, this sequence converges to 
\[
\left(\left(\prod_{i=1}^k x_i^{1/k}\right)^k x^{}_{k+1}\right)^{1/(k+1)}=\prod_{i=1}^{k+1} x_i^{1/(k+1)},
\]
that is, the $\tilde V_\ell$ converge to $V_{k+1}$ and hence the $B_\ell$ converge to $A_{k+1}$. This concludes the proof.
\end{proof}

Based on Proposition~\Cref{prop:makesymmetric} we obtain a proof for Theorem~\ref{thm:main2} for general real rank-two  directly from \Cref{thm:main}.

\begin{proof}[Proof of Theorem~\ref{thm:main2}]
Again, since $A$ has rank at most two, it suffices to prove the statement for general (i.e. nonsymmetric) $2 \times \dots \times 2$ tensors. Obviously, the minimal ratio $\| A \|_\sigma / \| A \|_F$ over general $2 \times \dots \times 2$ tensors is smaller or equal than the minimum over symmetric ones. However, by~\Cref{prop:makesymmetric} the converse is also true. The result hence follows from~\Cref{thm:main}. 
\end{proof}

Theorem~\ref{thm:main2} suggests an interesting relation between  results in~\cite{CKP_Extreme_00} and~\cite{KP_embedding_06}. The authors in~\cite{CKP_Extreme_00} found that the minimal possible ratio of spectral and Frobenius norm among all tensors in $\C^2\otimes \C^2\otimes \C^2$ is $\frac{2}{3}$, while in~\cite{KP_embedding_06} it is shown that the minimal ratio for tensors in $\R^2\otimes \R^2\otimes \R^2$  is only $\frac{1}{2}$. However, Theorem~\ref{thm:main2} states that border rank-two tensors in $\R^{2\times 2\times 2}$ have the minimal ratio $\frac{2}{3}$. This might be related to the fact that tensors of real rank two and three both have positive volume in $\R^2\otimes \R^2\otimes \R^2$, while almost all tensors in $\C^2\otimes \C^2\otimes \C^2$  have complex rank two.

\bibliographystyle{plain}
\bibliography{main}
\end{document}